\documentclass[british,english]{amsart}

\usepackage{mathptmx}

\usepackage[LGR,T1]{fontenc}
\usepackage{textcomp}
\usepackage[latin9]{luainputenc}
\usepackage{babel}

\usepackage{amstext}
\usepackage{amsthm}
\usepackage{amssymb}

\usepackage{tikz-cd}
\usepackage{tikz}

\usepackage{mathtools}

\usepackage{xcolor}

\usepackage[pdfusetitle,
 bookmarks=true,bookmarksnumbered=false,bookmarksopen=false,
 breaklinks=false,pdfborder={0 0 1},backref=false,colorlinks=false]
 {hyperref}

\makeatletter

\numberwithin{equation}{section}
\numberwithin{figure}{section}
\theoremstyle{plain}
\newtheorem{thm}{\protect\theoremname}[section]

\theoremstyle{definition}
\newtheorem{defn}[thm]{\protect\definitionname}
\newtheorem{example}[thm]{\protect\examplename}

\theoremstyle{plain}
\newtheorem{prop}[thm]{\protect\propositionname}
\newtheorem{lem}[thm]{\protect\lemmaname}
\newtheorem{fact}[thm]{\protect\factname}

\theoremstyle{plain}
\newtheorem*{lem*}{\protect\lemmaname}

\@ifundefined{date}{}{{\date{}}}

\makeatother

\addto\captionsbritish{\renewcommand{\definitionname}{Definition}}
\addto\captionsbritish{\renewcommand{\examplename}{Example}}
\addto\captionsbritish{\renewcommand{\factname}{Fact}}
\addto\captionsbritish{\renewcommand{\lemmaname}{Lemma}}
\addto\captionsbritish{\renewcommand{\propositionname}{Proposition}}
\addto\captionsbritish{\renewcommand{\theoremname}{Theorem}}
\addto\captionsbritish{\renewcommand{\corollaryname}{Corollary}}
\addto\captionsbritish{\renewcommand{\claimname}{Claim}}

\addto\captionsenglish{\renewcommand{\definitionname}{Definition}}
\addto\captionsenglish{\renewcommand{\examplename}{Example}}
\addto\captionsenglish{\renewcommand{\factname}{Fact}}
\addto\captionsenglish{\renewcommand{\lemmaname}{Lemma}}
\addto\captionsenglish{\renewcommand{\propositionname}{Proposition}}
\addto\captionsenglish{\renewcommand{\theoremname}{Theorem}}
\addto\captionsenglish{\renewcommand{\corollaryname}{Corollary}}
\addto\captionsenglish{\renewcommand{\claimname}{Claim}}

\providecommand{\corollaryname}{Corollary}
\providecommand{\claimname}{Claim}
\providecommand{\definitionname}{Definition}
\providecommand{\examplename}{Example}
\providecommand{\factname}{Fact}
\providecommand{\lemmaname}{Lemma}
\providecommand{\propositionname}{Proposition}
\providecommand{\theoremname}{Theorem}

\begin{document}
\selectlanguage{british}%


\global\long\def\res{\!\restriction}%
\global\long\def\actson{\curvearrowright}%

\global\long\def\Ad{\text{Ad}}%
\global\long\def\Rad{\text{Rad}}%
\global\long\def\FC#1{\text{FC}\left(#1\right)}%

\global\long\def\Prob#1{\text{Prob}\left(#1\right)}%
\global\long\def\Ch#1{\text{Ch}\left(#1\right)}%
\global\long\def\Tr#1{\text{Tr}\left(#1\right)}%
\global\long\def\Trleq#1{\text{Tr}_{\leq1}\left(#1\right)}%

\global\long\def\normalizer#1#2{\text{N}_{#1}\left(#2\right)}%
\global\long\def\Sub#1{\text{Sub}\left(#1\right)}%
\global\long\def\conn#1{#1^{0}}%
\global\long\def\explain#1#2{\underset{\underset{\mathclap{#2}}{\downarrow}}{#1}}%

\global\long\def\Ind#1#2{\text{Ind}_{#2}^{#1}}%
\global\long\def\Res#1#2{\text{Res}_{#2}^{#1}}%

\global\long\def\ev{\mathrm{ev}}%
\global\long\def\N{\mathbb{N}}%
\global\long\def\Z{\mathbb{Z}}%
\global\long\def\Q{\mathbb{Q}}%
\global\long\def\R{\mathbb{R}}%
\global\long\def\C{\mathbb{C}}%
\global\long\def\H{\mathbb{H}}%
\global\long\def\T{\mathbb{T}}%
\global\long\def\P{\mathbb{P}}%
\global\long\def\K{\mathbb{K}}%
\global\long\def\F{\mathbb{F}}%

\global\long\def\bG{\mathbf{G}}%
\global\long\def\bP{\mathbf{P}}%
\global\long\def\bH{\mathbf{H}}%
\global\long\def\bR{\mathbf{R}}%
\global\long\def\bQ{\mathbf{Q}}%
\global\long\def\bV{\mathbf{V}}%
\global\long\def\cP{\mathcal{P}}%
\global\long\def\bN{\mathbf{N}}%
\global\long\def\bM{\mathbf{M}}%
\global\long\def\bY{\mathbf{Y}}%
\global\long\def\bX{\mathbf{X}}%
\global\long\def\bB{\mathbf{B}}%
\global\long\def\bA{\mathbf{A}}%
\global\long\def\bT{\mathbf{T}}%
\global\long\def\bF{\mathbf{F}}%
\global\long\def\bL{\mathbf{L}}

\global\long\def\GL{\mathrm{GL}}%
\global\long\def\BH{\mathcal{B}(\mathcal{H})}%
\global\long\def\sl#1{\mathrm{SL}(#1)}%
\global\long\def\slz{SL_{n}\left(\mathbb{Z}\right)}%
\global\long\def\slr{SL_{n}\left(\mathbb{R}\right)}%
\global\long\def\Cstar{\mathrm{C}^{*}}%
\global\long\def\Cr{\mathrm{C}^{*}_{r}}%

\title{Selfless reduced $C^{*}$-algebras of linear groups}
\thanks{This work was supported by NSF postdoctoral fellowship grant DMS-2402368.}
\author{Itamar Vigdorovich}
\begin{abstract}
It is shown that the reduced $C^{*}$-algebra of a nontrivial linear
group $\Gamma\leq\mathrm{GL}_{d}(k)$ with trivial amenable radical
is selfless. Thus selflessness and simplicity coincide for reduced $C^*$-algebras of linear groups.  Similar results are obtained for twisted reduced group $C^{*}$-algebras.
\end{abstract}

\maketitle
\section{Introduction}
Let $\Gamma$ be a discrete group and $\lambda:\Gamma\to \mathrm U(\ell^2(\Gamma))$ its left regular representation. Then $\lambda$ extends linearly to a representation of the group algebra
$\lambda:\C[\Gamma]\to \mathcal B(\ell^2(\Gamma))$. The reduced $C^*$-algebra of $\Gamma$ is defined as the operator-norm closure
\[
C_r^*(\Gamma)=\overline{\lambda(\C[\Gamma])}\subset \mathcal B(\ell^2(\Gamma)).
\]
Reduced group $C^*$-algebras have long been a central source of important examples in the general study of $C^*$-algebras. In the other direction, understanding
$C^*$-algebraic properties of $C_r^*(\Gamma)$ is central to harmonic analysis on $\Gamma$.
For example, simplicity of $C_r^*(\Gamma)$ (that is, the lack of non-trivial two-sided ideals) is equivalent to $\lambda$ being weakly equivalent to every tempered representation.

The main focus of this paper is the notion of \emph{selfless $C^*$-algebras}, which is a strengthening of simplicity.
This notion was introduced by Robert in \cite{robert2025selfless}.
Rather than recalling the definition in full generality, which involves additional
terminology, we restrict attention to reduced group $C^*$-algebras.

We say that $C_r^*(\Gamma)$ is selfless if there exists a free ultrafilter $\mathcal U$ and a
$*$-homomorphism\footnote{$f$ is automatically injective and trace-preserving provided $\Gamma\ne \{e\}$.}
$
f \colon C_r^*(\Gamma * \Z) \to C_r^*(\Gamma)^{\mathcal U}
$
such that the following diagram commutes:
\[
\begin{tikzcd}\label{diagam}
	C_r^*(\Gamma) && C_r^*(\Gamma)^{\mathcal U} \\
	& C_r^*(\Gamma * \Z)
	\arrow["\Delta", hook, from=1-1, to=1-3]
	\arrow["\iota", hook, from=1-1, to=2-2]
	\arrow["f", hook, from=2-2, to=1-3]
\end{tikzcd}
\]
Here $C^*_r(\Gamma)^\mathcal U$ is the ultrapower $C^*$-algebra consisting of all bounded sequences quotiented by $\mathcal U$-null sequences, $\Delta$ denotes the diagonal embedding, and $\iota$ is the map induced by the canonical inclusion of
$\Gamma$ as a free factor of $\Gamma * \Z$.

Selflessness implies several of the most significant regularity properties in the theory
of $C^*$-algebras.
Among these, \emph{strict comparison} plays a central role: in the amenable (nuclear) setting it
is predicted to be equivalent to other key regularity properties via the Toms--Winter conjecture, while
in the non-amenable setting it is widely regarded as the most appropriate regularity property in the context of classification.

Another key consequence of selflessness is \emph{stable rank one}.
Together with strict comparison, these properties yield substantial simplifications of invariants such as the Cuntz
semigroup and $K$-theory, cf.
\cite{rieffelmain,2024strictcomparisonreducedgroup,schafhauser2025nuclear,gardella2022moderntheorycuntzsemigroups}.
A major class of examples to which the results of this paper apply consists of
fundamental groups of closed locally symmetric Riemannian manifolds.
In this geometric context, the resulting structural consequences are particularly appealing in
view of their connections to complex vector bundles, especially in the presence of the Baum--Connes conjecture for linear
groups, which is currently known only in certain special cases
\cite{valette2002introduction,ramagge1998haagerup,lafforgue}.

What makes selflessness powerful is not only its implications, but also its relatively simple categorical nature. Indeed, the von Neumann algebraic analogue of Diagram~(\ref{diagam}) holds for all $\mathrm{II}_1$-factors by Popa's theorem \cite{popa1995free}, and this was in fact a main motivation for Robert's definition of selflessness.
Furthermore, the key insight in \cite{2024strictcomparisonreducedgroup} is that the group-theoretic analogue of Diagram~(\ref{diagam}) recovers the well-studied notion of
mixed-identity-free groups.
It is this perspective that enabled the authors of \cite{2024strictcomparisonreducedgroup} to resolve the
long-standing open problem of strict comparison for $C_r^*(\F_n)$, as a consequence of selflessness.
Subsequently, selflessness has become an influential concept in the study of
$C^*$-algebras, leading to a rapidly growing body of work
\cite{elayavalli2025some,vigdorovich2025structural,raum2025strict,hayes2025selfless1,
hayes2025selfless2,houdayer2025selfless,flores2025selfless,ozawa2025proximality,avni2025mixed}.

Importantly, any selfless $C^*$-algebra is simple \cite[Theorem~3.1]{robert2025selfless}.
Whether the converse holds for reduced group $C^*$-algebras is a timely open problem
\cite[Problem~XCI]{schafhauser2025nuclear}.
The main purpose of this paper is to show that this converse does hold for the rich and important
class of linear groups.

\begin{defn}
A group $\Gamma$ is called \textit{linear} if it admits a faithful representation
$\Gamma\hookrightarrow\mathrm{GL}_{d}(k)$ for some field $k$ and some $d\in\N$.
\end{defn}

\begin{thm}\label{thm:main}
For a linear group $\Gamma\ne\{e\}$, the following are equivalent:
\begin{enumerate}
\item $C_{r}^{*}(\Gamma)$ is selfless.
\item $C_{r}^{*}(\Gamma)$ is simple.
\item $C_{r}^{*}(\Gamma)$ has a unique tracial state.
\item $\Gamma$ has no nontrivial amenable normal subgroups.
\end{enumerate}
\end{thm}

The implication $(1)\Rightarrow(2)$ holds in general \cite[Theorem~3.1]{robert2025selfless}, and so do the implications
$(2)\Rightarrow(3)\Rightarrow(4)$, which are standard by this point (see e.g.\ \cite{breuillard2017c}). Moreover, for linear groups the implication
$(4)\Rightarrow(2)$ was proved in the celebrated work \cite{breuillard2017c}.

Our contribution is the implication $(4)\Rightarrow(1)$.
 This is made possible by the recent remarkable work of Ozawa \cite{ozawa2025proximality}, where (among other results) a dynamical criterion for selflessness of
reduced group $C^*$-algebras, called property $\mathrm{P}_{\mathrm{PHP}}$, is established (see Definition \ref{def:php}).
Ozawa verified $\mathrm{P}_{\mathrm{PHP}}$ for Zariski-dense subgroups of $\mathrm{PSL}_{n}(\R)$, and here we  show that $\mathrm{P}_{\mathrm{PHP}}$
holds for Zariski-dense subgroups of general semisimple $S$-algebraic groups. This includes certain
groups that are not strictly linear, but are ``almost linear'' in the $S$-adic sense.

\begin{thm}\label{thm:s-adic}
    Let $S$ be a finite set such that to each $v\in S$ is associated a local field $\K_{v}$ and a connected  adjoint  $\K_{v}$-simple algebraic group $\bG_{v}$. Let $G=\prod_{v\in S}\bG_{v}(\K_{v})$ and let
$\Gamma<G$ be a subgroup whose projection to each $\bG_v(\K_v)$ is Zariski dense and unbounded. Then $C^*_r(\Gamma)$ is selfless.
\end{thm}

Other  notable examples of linear groups include braid groups, virtually special groups,
$S$-arithmetic groups, and more generally thin groups. In a  work in progress \cite{theartofbuilding},  (non-linear) groups acting on exotic buildings are covered. Another class for which the four conditions of
Theorem~\ref{thm:main} are equivalent is that of acylindrically hyperbolic groups
\cite{ozawa2025proximality,yang2025extreme,2024strictcomparisonreducedgroup}.

It is worth mentioning that the equivalence of (3) and (4) (for arbitrary groups) remained open for quite a while, until settled in the affirmative
in \cite{breuillard2017c}. The equivalence of (2) and (3) was also open for a long time, until settled in the negative in \cite{le2017c}.
It remains to be seen whether (1) and (2) are equivalent in general, or whether interesting group constructions will yield a separation.

\smallskip

We conclude by stating a more general result regarding twisted reduced group
$C^*$-algebras (see \cite{bedos2023c} for the definition),  which follows from our work together with the recent \cite[Theorem~A]{flores2026pureness}.

\begin{thm}\label{thm: twisted}
Let $\Gamma\ne\{e\}$ be a linear group with trivial amenable radical or a group as in the statement of Theorem \ref{thm:s-adic}. Then for any $2$-cocycle $\omega\in Z^2(\Gamma,\T)$, the twisted reduced group
$C^*$-algebra $C_r^*(\Gamma,\omega)$ is selfless.
\end{thm}

We further note that our results yield \emph{complete selflessness} in the sense of \cite{ozawa2025proximality}.

\smallskip

\subsection*{Acknowledgments}

The author thanks Corentin Le Bars, Emmanuel Breuillard, and Alon Dogon
for helpful correspondence. The author is also grateful to the referee for their careful reading and helpful remarks.

\section{Notations and preliminaries}\label{sec:prem}

We assume basic familiarity with algebraic groups. We nevertheless include examples that should allow a non-specialist to follow the general theme.

Let $k$  be a field, with algebraic closure $\bar{k}$.
For a $k$-algebraic variety $\bV$ and a field extension $\tilde{k}/k$, we denote by
$\bV(\tilde{k})$ the set of $\tilde{k}$-points. In the special case $\tilde{k}=\bar{k}$
we identify $\bV$ with $\bV(\bar{k})$.

Recall that a (not necessarily Hausdorff) topological space $X$ is \emph{irreducible} if it is
non-empty and cannot be written as a union of two proper closed subsets. Equivalently, any two
non-empty open subsets of $X$ intersect. The product of finitely many irreducible topological spaces
is irreducible. Moreover, a dense subset of an irreducible space is irreducible with the subspace topology.

\begin{lem}[Zariski density of $k$-points]\label{lem:kpoints-dense}
Assume that  $k$ is infinite,  and let $\bG$ be a connected reductive $k$-algebraic group.
Then $\bG(k)$ is Zariski dense in $\bG$. In particular, the topological space $\bG(k)$ endowed with
its Zariski topology is irreducible.
\end{lem}

\begin{proof}
Since $\bG$ is reductive and connected,
$\bG(k)$ is Zariski dense in $\bG$ \cite[Cor.\ 18.3]{borel2012linear}.

The second assertion follows as any connected $k$-algebraic group is
irreducible as a $k$-variety, and $\bG(k)$ is Zariski dense in $\bG$ by the first part; hence
$\bG(k)$ is irreducible as it is dense in $\bG$. Thus, in what follows, reductivity is used only through the density statement above, while the irreducibility conclusion is a consequence of connectedness together with this density.
\end{proof}

Let $S$ be a finite set, and for each $v\in S$ fix a field $k_v$. By an $S$-algebraic variety we mean a product of the form
\[
V_S:=\prod_{v\in S}\bV_v(k_v)
\]
where each $\bV_v$ is a $k_v$-algebraic variety.  Formally, this notion depends on the fields $k_v$ attached to the points $v\in S$, but by abuse of notation, we assume that $k_v$ are implied by the notation $S$.
We endow $V_S$ with the \emph{$S$-Zariski topology},
namely the product topology of the Zariski topologies on the factors $\bV_v(k_v)$.  Note that even if all $k_v$ are
equal to the same field $k$, so that $\bV:=\prod_{v\in S}\bV_v$ is a $k$-variety and $V_S=\bV(k)$,
the $S$-Zariski topology on $V_S$ need not coincide with the Zariski topology of the $k$-variety $\bV$ (the latter is typically finer).

Let us now consider \emph{$S$-algebraic groups}. Thus, for each $v\in S$
let $\bG_v$ be a $k_v$-algebraic group and consider the product group
\[
G:=\prod_{v\in S}\bG_v(k_v),
\]
equipped with its $S$-Zariski topology. We will say that $G$ satisfies a certain property of algebraic groups (e.g. connected, adjoint, reductive..) if each of the $\bG_v$ satisfies this property.
For $v\in S$ we denote by $\pi_v:G\to \bG_v(k_v)$ the
coordinate projection. A subset $\Gamma\subset G$ is \emph{$S$-Zariski dense} if it is dense in the
$S$-Zariski topology. An \emph{$S$-algebraic
subgroup} of $G$ means a subgroup of the form $\prod_{v\in S}\bH_v(k_v)$ with $\bH_v\le \bG_v$
a $k_v$-algebraic subgroup.

\begin{lem}\label{lem:S-Zariski-density}
Assume that each field $k_v$ is infinite, and that each $\bG_v$ is connected and reductive.  For a subgroup $\Gamma\le G$, the following are equivalent:
\begin{enumerate}
    \item[\rm(i)] $\Gamma$ is $S$-Zariski dense in $G$.
    \item[\rm(ii)] If $H=\prod_{v\in S}\bH_v(k_v)\le G$ is an $S$-algebraic subgroup containing $\Gamma$,
    then $H=G$.
    \item[\rm(iii)] For every $v\in S$, the projection $\pi_v(\Gamma)$ is Zariski dense in $\bG_v(k_v)$.
\end{enumerate}
\end{lem}

In the proof below, reductivity is used only through Lemma~\ref{lem:kpoints-dense}, namely to ensure that the relevant spaces of $k_v$-points are irreducible in the Zariski topology.

\begin{proof}
{\rm(i)$\Rightarrow$(iii).} Each $\pi_v$ is continuous for the product topology, hence sends dense sets
to dense sets.

{\rm(iii)$\Rightarrow$(ii).} If $\Gamma\subseteq \prod_{v\in S}\bH_v(k_v)$, then
$\pi_v(\Gamma)\subseteq \bH_v(k_v)$ for all $v$. Since $\pi_v(\Gamma)$ is Zariski dense in
$\bG_v(k_v)$ and $\bH_v(k_v)$ is Zariski closed in $\bG_v(k_v)$, we get $\bH_v(k_v)=\bG_v(k_v)$,
hence $H=G$.

{\rm(ii)$\Rightarrow$(iii).} Fix $v_0\in S$ and let $\bH_{v_0}\le \bG_{v_0}$ be the Zariski closure
of $\pi_{v_0}(\Gamma)$ in $\bG_{v_0}$. Then $\bH_{v_0}$ is a $k_{v_0}$-algebraic subgroup and
\[
\Gamma\subseteq \bH_{v_0}(k_{v_0})\times \prod_{v\in S\setminus\{v_0\}}\bG_v(k_v),
\]
so (ii) forces $\bH_{v_0}=\bG_{v_0}$.

{\rm(iii)$\Rightarrow$(i).} It suffices to show that $\Gamma$ meets every  basic open set
$\prod_{v\in S}U_v$, where each $U_v\subseteq \bG_v(k_v)$ is a non-empty Zariski open subset.
We argue by induction on $|S|$. The case $|S|=1$ is exactly (iii).

Assume $|S|\ge 2$ and fix $v_0\in S$. Write $S=S'\sqcup\{v_0\}$ and set
\[
X:=\prod_{v\in S'}\bG_v(k_v),\qquad Y:=\bG_{v_0}(k_{v_0}),\quad \text{so }\quad G=X\times Y.
\]
Let $U\subseteq X$ and $V\subseteq Y$ be non-empty  open subsets.
Using density of $\pi_{v_0}(\Gamma)$, for each $y\in Y$ the open set $V^{-1}y$ meets $\pi_{v_0}(\Gamma)$,
hence
\[
Y=\bigcup_{b\in \pi_{v_0}(\Gamma)} Vb.
\]
The Zariski topology on $Y$ is Noetherian, hence quasi-compact, so there exist
$b_1,\dots,b_n\in \pi_{v_0}(\Gamma)$ such that $Y=\bigcup_{j=1}^n Vb_j$.
Choose $c_j\in X$ such that $\delta_j=(c_j,b_j)\in \Gamma$.

For each $j$ define the ``cut-and-project'' set
\[
A_j:=\pi_X\!\big(\Gamma\cap (X\times Vb_j)\big)\subseteq X,
\]
so that $\pi_X(\Gamma)=\bigcup_{j=1}^n A_j$.
By (iii) the subgroup $\pi_X(\Gamma)\le X$ has Zariski dense projections to each factor indexed by $S'$,
hence by induction it is $S'$-Zariski dense in $X$. Therefore
\[
X=\overline{\pi_X(\Gamma)}= \bigcup_{j=1}^n \overline{A_j}.
\]
By Lemma~\ref{lem:kpoints-dense}, each factor $\bG_v(k_v)$ ($v\in S'$) is irreducible, hence so is $X$ (this is the only place where we use that $\bG_v$ is reductive).
Thus one of the closed sets $\overline{A_j}$ equals $X$; fix $j_0$ with $\overline{A_{j_0}}=X$.

Right-translation by $c_{j_0}^{-1}$ is a homeomorphism of $X$, so $A_{j_0}c_{j_0}^{-1}$ is dense in $X$
and meets $U$. Choose $x\in A_{j_0}$ such that $x c_{j_0}^{-1}\in U$.
By definition of $A_{j_0}$, there exists $y\in Vb_{j_0}$ with $(x,y)\in\Gamma$.
Then $(x,y)\delta_{j_0}^{-1}\in \Gamma$ has $Y$-coordinate $y b_{j_0}^{-1}\in V$ and $X$-coordinate
$x c_{j_0}^{-1}\in U$, so $\Gamma\cap (U\times V)\neq\varnothing$. This proves (i).
\end{proof}

We finish this section with the following elementary  lemma on group actions.
\begin{lem}\label{lem:Machnisilizer}
Let $G$ be a group acting transitively on a set $X$. Assume that the point-stabilizers are self-normalizing, i.e. $G_x =\normalizer{G}{G_x}$ where $G_x$ is the stabilizer of a point $x\in X$.
For $c\in G $ and $x,y\in X$, if $cgx=gy$ for all $g\in G$ then $c\in \ker(G\actson X)$, and  $x=y$.
\end{lem}

\begin{proof}
Assume $cgx=gy$ for all $g\in G$. Then in particular for $g=e$ we get  $y=cx$. Hence $cgx=gcx$, which in turn implies that $c$ is in the normalizer of $G_x$. By assumption, $c$ must in fact belong to $G_x$, so we have $cgx=gx$ for $g\in G$. Transitivity then implies that $c$ acts trivially on $X$. It then also follows that $x=y$.
\end{proof}

\section{Geometry of flag spaces\label{sec:geometry-of-flag}}

Fix an arbitrary field $k$, and connected reductive (and non-trivial) $k$-algebraic group $\bG$.
Fix also a proper $k$-parabolic subgroup $\bQ<\bG$, which means (by definition) that the homogeneous
variety $\bB:=\bG/\bQ$ is a projective $k$-variety, not reduced to a single point.
As $\bQ$ is its own normalizer \cite[Theorem 11.16]{borel2012linear}, $\bB$ can be identified with
the space of all conjugates of $\bQ$.

Two points $x,y\in\bB$ are said to be \textit{transversal} (often called opposite) if the intersection
of the corresponding conjugates of $\bQ$ (equivalently, of the stabilizer subgroups of $x$ and $y$)
is a common Levi subgroup. We denote
\[
\bY_{x}=\left\{ y\in\bB:\ x\text{ and }y\text{ are not transversal}\right\}.
\]
By the Bruhat decomposition \cite[\S14.20]{borel2012linear}, $\bY_{x}$ is a closed proper subvariety
of $\bB$, which is moreover defined over $k$ if $x\in\bB(k)$.

\begin{example}\label{exa:flags}
Let $V$ be a complex vector space of dimension $d\in\N$. A \textit{full flag} in $V$ is a tuple
$(V_{1},...,V_{d-1})$ consisting of linear subspaces of $V$ and satisfying
\[
\{0\}\subsetneq V_{1}\subsetneq...\subsetneq V_{d-1}\subsetneq V.
\]
For example, fixing a basis $e_{1},...,e_{d}$ one gets the ``standard flag'' by setting
$V_{i}=\mathrm{Sp}\{e_{1},...,e_{i}\}$. Let $B$ denote the space of all full flags in $V$.
The group $\mathrm{SL}_{d}(\C)$ acts on $B$ transitively, and the kernel of this action is the subgroup
of scalar matrices. We thus get a transitive faithful action of $G=\mathrm{PSL}_{d}(\C)$ on $B$.
The stabilizer of the standard flag is the subgroup $Q\leq G$ consisting of all upper triangular
matrices, so that we get an identification $B=G/Q$.
Elements $x,y\in B$ are transversal (in the sense defined above) if and only if the corresponding full
flags  $(V_{i})_{i}$ and $(W_{i})_{i}$ satisfy
$V_{i}\cap W_{d-i}=\{0\}$ for each $i=1,...,d-1$.
\end{example}

\begin{defn}\label{def:general-position}
A set $E\subset \bB$ is said to be in \emph{general position} if for every finite subset $E_0\subset E$
and every $x'\in E\setminus E_0$, whenever $\bigcap_{x\in E_0}\bY_x\neq\emptyset$ one has
$\bigcap_{x\in E_0}\bY_x\nsubseteq \bY_{x'}$.
\end{defn}

Often $E$ will be presented as a family $E=(x_i)_{i\in I}$. In this situation we interpret $E$ in the
multiset sense; in particular, we tacitly assume that the points $x_i$ are pairwise distinct.

\smallskip
We also need an $S$-adic analogue of Definition~\ref{def:general-position}. Let $S$ be a finite set and,
for each $v\in S$, let $k_v$ be a field, let $\bG_v$ be a connected reductive $k_v$-algebraic group, and
let $\bQ_v<\bG_v$ be a proper $k_v$-parabolic subgroup. Set
\[
\bB_{v}:=\bG_{v}/\bQ_{v},
\qquad
\mathcal{B}_{v}:=\bB_{v}(k_{v}).
\]
Equivalently, $\mathcal{B}_{v}$ identifies with the $\bG_{v}(k_v)$-homogeneous space
$\bG_{v}(k_v)/\bQ_{v}(k_v)$ (cf.\ \cite[Proposition 20.5]{borel2012linear}).
Consider the $S$-algebraic group $G=\prod_{v\in S}\bG _v(k_v)$ and the associated flag space
\[
\mathcal{B}=\mathcal{B}_S:=\prod_{v\in S}\mathcal{B}_{v}.
\]

\begin{defn}\label{def:general-position s-adic}
A set $E\subset \mathcal{B}$ is said to be in \emph{general position} if, for each $v\in S$, the
(multiset of) $v$-coordinates
\[
E_v:=\{x_v\}_{x\in E}\subset \mathcal{B}_v
\]
is in general position in the sense of Definition~\ref{def:general-position}.
\end{defn}

For $x_{v}\in\mathcal{B}_{v}$ let $\bY_{x_{v}}\subset\bB_{v}$ be the (proper) closed subvariety of
non-transversal flags as described above, and denote
\[
Y_{v}(x_{v}):=\bY_{x_{v}}(k_{v})\subset\mathcal{B}_{v}.
\]
For $x=(x_{v})_{v}\in\mathcal{B}_{S}$ define
\begin{align}\label{eq:Yx}
\widehat{Y}_{x_{v}}
:=Y_{v}(x_{v})\times\prod_{v'\in S\setminus\{v\}}\mathcal{B}_{v'},
\qquad
Y_x:=\bigcup_{v\in S}\widehat{Y}_{x_{v}}\subset\mathcal{B}.
\end{align}

The following proposition is the main geometric input of the paper. It is a uniform Noetherian statement for families of flags in general position: once sufficiently many flags are chosen in general position, the corresponding non-transversality loci have empty total intersection.

\begin{prop}\label{prop:large-intersections-general}
There exists a constant $K=K(G)$ with the following property.
If $E\subset \mathcal B_S$ is in general position then
\[
\bigcap_{x\in E_0} Y_x=\emptyset
\qquad\text{for every }E_0\subset E\text{ with }|E_0|>K.
\]
\end{prop}
We first need the following elementary uniform Noetherian lemma.

\begin{lem}\label{lem:noetherianity}
For any $d,D>0$ there exists $K>0$ such that the following holds.
Let $\P^{d}$ be the $d$-dimensional projective space over a field $k$, and suppose that
\[
\P^{d}\supset \bX_{0}\supset \bX_{1}\supset \cdots \supset \bX_{K}
\]
is a descending chain of closed $k$-subvarieties, each of which is defined as the zero locus of a set
of homogeneous polynomials of degree $\le D$. Then the chain is improper, namely
$\bX_i=\bX_{i-1}$ for some $i=1,\dots,K$.
\end{lem}

\begin{proof}
Let $W_r$ denote the $k$-vector space of homogeneous polynomials of degree $r$ in $d+1$ variables and
set $W_{\le D}:=\bigoplus_{r=0}^D W_r$. Let $K:=\dim(W_{\le D})+1$.

Write $\bX_i=\mathcal V(A_i)$ for some finite $A_i\subset W_{\le D}$, and replace $A_i$ by
$\bigcup_{j\le i}A_j$ so that $A_0\subset\cdots\subset A_K$.
Let
\[
U_i:=\bigoplus_{r=0}^D \mathrm{Sp}(A_i\cap W_r)\ \le\ W_{\le D}.
\]
Then $U_0\subset\cdots\subset U_K$ is an ascending chain of graded subspaces of a space of dimension $K-1$, hence
$U_i=U_{i-1}$ for some $i$. Therefore
\[
\bX_i=\mathcal V(A_i)=\mathcal V(U_i)=\mathcal V(U_{i-1})=\mathcal V(A_{i-1})=\bX_{i-1},
\]
as desired.
\end{proof}

\begin{proof}[Proof of Proposition \ref{prop:large-intersections-general}]
We begin with the case $|S|=1$.
Let $S=\{v\}$ and write $\bG=\bG_v$, $\bQ=\bQ_v$, $\bB=\bB_v$, and $Y_x=\bY_x(k_v)$.
Choose an immersive representation $\bG\to \mathrm{GL}_{d+1}$ for which $\bQ$ is the stabilizer of a line
(see \cite[Thm.~5.1]{borel2012linear}). This yields a $\bG$-equivariant immersion
$\bB\hookrightarrow \P^d$ . This immersion is in fact a closed embedding because  $\bB$ is projective.

Fix $x_0\in \bB$ and write $\bY_{x_0}=\mathcal V(f_1,\dots,f_r)$ inside $\P^d$ for homogeneous
polynomials $f_i$, and set $D:=\max_i \deg(f_i)$.
For $x=gx_0$ we have $\bY_x=g\bY_{x_0}$, and since the action on $\P^d$ is by projective linear
transformations, $\bY_x$ is cut out by equations of degree $\le D$, uniformly in $x$.

Let $K(d,D)$ be as in Lemma~\ref{lem:noetherianity}.
Consider $E\subset \bB(k_v)$  in general position and $E_0=\{x_1,\dots,x_m\}\subset E$ with $m>K(d,D)$. Set
$\bX_\ell:=\bigcap_{j=1}^\ell \bY_{x_j}$. Then $\bX_1\supset\cdots\supset \bX_m$ is a descending chain
of closed subvarieties of $\P^d$ each defined by equations of degree $\le D$, so Lemma~\ref{lem:noetherianity}
gives $\bX_\ell=\bX_{\ell-1}$ for some $\ell$. Hence
$\bigcap_{j=1}^{\ell-1}\bY_{x_j}\subset \bY_{x_\ell}$, and general position forces
$\bigcap_{j=1}^{\ell-1}\bY_{x_j}=\emptyset$, so in particular $\bigcap_{j=1}^m \bY_{x_j}=\emptyset$.

The resulting bound depends on $\bG$ and the conjugacy class of $\bQ$, but since $\bG$ has only finitely
many parabolic subgroups up to conjugation, we may take the maximum over all conjugacy classes and obtain
a constant $K_v=K(\bG_v)$ that works for every choice of $\bQ_v$.

Now assume $S$ is arbitrary. For each $v\in S$ let $K_v=K(\bG_v)$ be as above and set
\[
K:=|S|\cdot \max_{v\in S} K_v.
\]
Let $E\subset \mathcal B_S$ be in general position and let $E_0\subset E$ with $|E_0|>K$.
If $y\in \bigcap_{x\in E_0} Y_x=\bigcap_{x\in E_0}\bigcup_{v\in S}\widehat Y_{x_v}$, and for each
$x\in E_0$ choose $v(x)\in S$ with $y\in \widehat Y_{x_{v(x)}}$.
By the pigeonhole principle there exist $v\in S$ and $\widetilde E_0\subset E_0$ with $|\widetilde E_0|> |E_0|/|S|$ such that
$v(x)=v$ for all $x\in \widetilde E_0$. Then $y_v\in \bY_{x_v}(k_v)$ for all $x\in \widetilde E_0$, so
\[
\bigcap_{x\in \widetilde E_0}\bY_{x_v}(k_v)\neq\emptyset,
\]
contradicting the $|S|=1$ case since $\{x_v\}_{x\in \widetilde E_0}$ is in general position and
$|\widetilde E_0|>K_v$.
\end{proof}

\section{Dynamics\label{sec:Dynamics}}

We now consider algebraic groups over local fields. Recall that a local field $\K$ is a locally compact, Hausdorff, non-discrete topological field.
By the classification of local fields, every such $\K$ is isomorphic  (as a topological field)  to
a finite field extension of $\R$, $\Q_p$ or $\F_p(\!(t)\!)$, for some prime $p$.

Let $S$ be a finite set. For each $v\in S$ let $\K_{v}$ be a local field, and let $\bG_{v}$ be a semisimple  $\K_{v}$-algebraic group. Let $\bG_v^\circ$ denote the identity connected components and   set
\[
G:=\prod_{v\in S}\bG_{v}(\K_{v}),\qquad G_c:=\prod_{v\in S}\bG^\circ_{v}(\K_{v}).
\]

All spaces of $\K_{v}$-points (and their finite products) are endowed with the Hausdorff topology coming from the local fields; this is the default topology in this section.
When we refer to the Zariski topology (or $S$-Zariski topology), we will say so explicitly.

\smallskip

We follow \cite[\S8 and \S9]{benoist2016random}.
For each $v\in S$ fix a maximal $\K_{v}$-split torus $\bA_{v}<\bG_{v}$ with Lie algebra
$\mathfrak{a}_{v}$, and consider the set $\Sigma_v$ of restricted roots. Fix a  choice of positive restricted roots $\Sigma_v^+$, and let $\Pi_{v}\subset \Sigma_v$ be the corresponding
set of simple restricted roots. Set
\[
\mathfrak{a}:=\prod_{v\in S}\mathfrak{a}_{v},
\qquad
\Pi:=\bigsqcup_{v\in S}\Pi_{v}.
\]
Let $\kappa:G\to\mathfrak{a}$ and $\lambda:G\to\mathfrak{a}$ be the Cartan and Jordan
projections \cite[\S8.7]{benoist2016random}.

\smallskip

Let $\Gamma<G$ be an $S$-Zariski dense  subgroup. Define $\Theta_{\Gamma}\subset\Pi$ by
\[
\Theta_{\Gamma}:=\bigl\{\alpha\in\Pi:\ \alpha^{\omega}(\kappa(\Gamma))\ \text{is not bounded}\bigr\}.
\]
For the precise definition of $\alpha^\omega$ see \cite[\S8.1]{benoist2016random}.
For each $v\in S$ we set
\[
\Theta_{\Gamma,v}:=\Theta_{\Gamma}\cap \Pi_{v},
\qquad\text{so that}\qquad
\Theta_{\Gamma}=\bigsqcup_{v\in S}\Theta_{\Gamma,v}.
\]
Let $\bQ_{\Theta_{\Gamma,v}}<\bG_{v}^\circ$ be the standard parabolic subgroup
of type $\Theta_{\Gamma,v}$ (as in \cite[\S8.6]{benoist2016random}). Then $\bQ_{\Theta_{\Gamma,v}}$ is connected and
$\normalizer{\bG_v^\circ}{\bQ_{\Theta_{\Gamma,v}}}=\bQ_{\Theta_{\Gamma,v}}$ \cite[Theorem 11.16]{borel2012linear}.
Set
\[
    \bQ_v:=\normalizer{\bG_v}{\bQ_{\Theta_{\Gamma,v}}}.
\]
As $\bG_v^\circ\lhd \bG_v$, we have
\[
\bQ_v^\circ \subset \bQ_v\cap \bG_v^\circ
=
\normalizer{\bG_v}{\bQ_{\Theta_{\Gamma,v}}}\cap \bG_v^\circ
=
\normalizer{\bG_v^\circ}{\bQ_{\Theta_{\Gamma,v}}}
=
\bQ_{\Theta_{\Gamma,v}}\subset \bQ_v^\circ,
\]
and therefore $\bQ_v^\circ=\bQ_{\Theta_{\Gamma,v}}$. Moreover, $\bQ_v$ is self-normalizing in $\bG_v$:
indeed, if $g\in \normalizer{\bG_v}{\bQ_v}$, then
\[
g\bQ_{\Theta_{\Gamma,v}}g^{-1}
=
g\bQ_v^\circ g^{-1}
=
(g\bQ_v g^{-1})^\circ
=
\bQ_v^\circ
=
\bQ_{\Theta_{\Gamma,v}},
\]
hence $g\in \normalizer{\bG_v}{\bQ_{\Theta_{\Gamma,v}}}=\bQ_v$. In particular,
\[
\normalizer{\bG_v}{\bQ_v}=\bQ_v.
\]

Set
\begin{align}\label{eq:boundary}
\bB_{v}:=\bG_{v}/\bQ_{v},
\qquad
\mathcal{B}_{v}:=\bB_{v}(\K_{v}), \qquad \mathcal{B}:=\prod_{v\in S}\mathcal{B}_{v}.
\end{align}
Let $F_v:=\bG_v/\bG_v^\circ$ and $F:=\prod_{v\in S}F_v$. By \cite[Lemma~9.1]{benoist2016random},
the set $\Theta_\Gamma$ is stable under the natural $F$-action. Consequently, the corresponding
parabolic type is $F_v$-invariant, and $\bB_v=\bG_v/\bQ_v$ is (canonically) identified with
$\bB_v^\circ=\bG_v^\circ/\bQ_v^\circ$; in particular,
\[
\mathcal B_v=\bB_v(\K_v)\ \cong\ \bB_v^\circ(\K_v)
\qquad\text{(see \cite[\S8.6]{benoist2014random}).}
\]
It follows that $\mathcal B=\prod_{v\in S}\mathcal B_v$ is the same flag space as in the connected
case, and hence the geometric results of the previous section apply to $\mathcal B$. In particular, $G_c$ acts transitively on $\mathcal B$.
Note that  each space $\mathcal{B}_v$ is compact because the $\K_v$-variety $\bB_v$ is projective. Hence $\mathcal B$ is compact.

\smallskip
We say that \(g\in G_c\) is \emph{\(\Theta_\Gamma\)-proximal} if
\[
    \alpha^\omega(\lambda(g))>0
    \qquad \text{for all } \alpha\in\Theta_\Gamma .
\]
This is equivalent to saying that the action of \(g\) on \(\mathcal B\) has a unique attracting
fixed point \(x_g^+\); see \cite[\S9.2]{benoist2016random} and
\cite[Def.~2.25]{gueritaud2017anosov}. In this context, ``attracting'' means that the fixed point $x_g^+$ admits a compact
neighborhood \(b^+\) such that $g^n x \rightarrow x_g^+$
uniformly for \(x\in b^+\) (this is clarified after Lemma 6.39 in \cite{benoist2016random}). Note that
in the terminology of \cite[Def.~2.25]{gueritaud2017anosov}, the same
dynamical property follows from the corresponding infinitesimal contraction
condition valid in the Archimedean case. We shall use the following slightly more precise form of this dynamics.
Recall the definition of the sets \(Y_x\) from \eqref{eq:Yx}.

\begin{fact}\label{fact:loxodromic}
Let $g\in G_c$ be $\Theta_\Gamma$-proximal. Then there exists a transverse pair
$(x_{g}^{+},x_{g}^{-})\in\mathcal{B}\times\mathcal{B}$ such that $g^nx\to x_{g}^+$ for any
$x \in \mathcal{B}\setminus Y_{x_{g}^{-}}$. This convergence is moreover uniform on compact subsets of 
$\mathcal{B}\setminus Y_{x_{g}^{-}}$.
\end{fact}

\begin{proof}
Let $g\in G_c$ be $\Theta_\Gamma$-proximal. For each $v\in S$, the corresponding statement for
$g_v\in\bG_v(\K_v)$ acting on $\mathcal B_v$ is standard.
Indeed, in the Archimedean case, pointwise convergence is stated explicitly in
\cite[Lemma~2.26]{gueritaud2017anosov},
and the proof there uses the representation-theoretic characterization of proximality
in terms of proximality in a suitable projective representation.
This characterization is available over arbitrary local fields as well; see
\cite[\S8--\S9]{benoist2014random} and, in particular, \cite[\S9.2]{benoist2016random} for the
dynamical picture on $\mathcal B_v$. 
Thus there exists a transverse pair $(x_{g_v}^+,x_{g_v}^-)\in
\mathcal B_v\times\mathcal B_v$ such that $g_v^n x \to x_{g_v}^+$ for every
$x\in \mathcal B_v\setminus Y_v(x_{g_v}^-)$. 

This pointwise convergence is upgraded to uniform convergence on compact subsets as follows. By definition of
$x_{g_v}^+$ being an attracting fixed point, it admits a compact neighborhood $b^+$ such that, for every neighborhood $U$ of
$x_{g_v}^+$, there exists $N_0$ with
$g_v^n(b^+)\subset U$ for all $n\ge N_0$.
Now let $C\subset \mathcal B_v\setminus Y_v(x_{g_v}^-)$ be compact. For each $x\in C$, pointwise convergence
gives $n_x$ such that $g_v^{n_x}x$ is in the interior of $b^+$. By continuity, there is a neighborhood $O_x$ of $x$
such that $g_v^{n_x}(O_x)$ is contained in the interior of $b^+$. By compactness of $C$, finitely many such $O_x$ cover $C$.
For the corresponding finitely many integers $n_x$, choose $N$ so large that $N-n_x\ge N_0$ for all of them. Then, for every
$n\ge N$ and every $y\in C$, choosing $x$ with $y\in O_x$ gives
\[
    g_v^n y = g_v^{n-n_x}(g_v^{n_x}y)\in U.
\]
Thus $g_v^n|_C$ converges uniformly to the constant map $x\mapsto x_{g_v}^+$.

Finally, the $S$-adic statement follows by taking products: the topology on $\mathcal B$ is the
product topology, and $Y_{x}$ is defined so that $\mathcal B\setminus Y_{x_g^-}$ is a union of
products of sets of the form $\mathcal B_v\setminus Y_v(x_{g_v}^-)$.
\end{proof}

We refer to $(x_{g}^{+},x_{g}^{-})$ as the \emph{attracting/repelling fixed points} of the proximal element $g$. It is worth mentioning that the fact that both points $x_{g}^{+},x_{g}^{-}$ lie in the same flag space $\mathcal B$ does not hold in the general setting of $\Theta$-proximality for arbitrary subsets of simple roots $\Theta$. Rather, it relies crucially on a certain symmetry of $\Theta_\Gamma$.

Indeed, let $\iota:\mathfrak a^+\to \mathfrak a^+$ be the opposition involution, defined by $\iota=-\mathrm{Ad}_{w_0}$, where $w_0$ is the longest element of the Weyl group with respect to the fixed choice of positive roots. We have
\[
\kappa(\gamma^{-1})=\iota(\kappa(\gamma)).
\]
Hence $\kappa(\Gamma)$ is $\iota$-stable, and $\Theta_{\Gamma}$ is unchanged when passing from $\Gamma$ to $\Gamma^{-1}$. In general, however, if $g$ is $\Theta$-proximal, then its attracting point lies in the flag space corresponding to $\Theta$, whereas its repelling point lies in the opposite flag space corresponding to $\iota(\Theta)$ (as explained in  \cite[Lemma 2.26]{gueritaud2017anosov}). This point is essential here, and it is also crucial for the following lemma.

\begin{lem}\label{lem:proximal-exists}
Assume that $\Gamma<G$ is $S$-Zariski dense and denote $\Gamma_c=\Gamma\cap G_c$.
\begin{enumerate}
\item $\Gamma_c$ contains $\Theta_{\Gamma}$-proximal elements.
\item If $\gamma\in\Gamma$ is $\Theta_{\Gamma}$-proximal with attracting/repelling fixed points
$(x^{+},x^{-})$, then $\gamma^{-1}$ is $\Theta_{\Gamma}$-proximal with attracting/repelling fixed
points $(x^{-},x^{+})$.
\item If a $\K_v$-simple factor  $\bH_v(\K_v)$ of $\bG_v^\circ(\K_v)$ is in the kernel of  $\bG_v(\K_v)\actson \mathcal B_v$ then the projection of $\Gamma_c$ into $\bH_v(\K_v)$ is bounded.
\end{enumerate}
\end{lem}

\begin{proof}
(1) This is \cite[Lemma~9.2]{benoist2016random}.

\smallskip
\noindent(2) As explained above, the identity $
\kappa(\gamma^{-1})=\iota(\kappa(\gamma))$
implies that $\Theta_{\Gamma}$ is invariant under the opposition involution $\iota$. Thus the attracting and repelling fixed points of a $\Theta_{\Gamma}$-proximal element both belong to the same flag space $\mathcal B$. Since the inverse dynamics interchange attracting and repelling directions, $\gamma^{-1}$ is again $\Theta_{\Gamma}$-proximal, with attracting/repelling fixed points $(x^{-},x^{+})$.

\smallskip
\noindent(3) Let \(v\in S\), and let \(\bH_v(\K_v)\) be a \(\K_v\)-simple factor of \(\bG_v^\circ\) which lies in the kernel of the action
\(\bG_v(\K_v)\actson \mathcal B_v\).
Then $\bH_v(\K_v)$ is contained in the stabilizer of each point in $\mathcal B_v$, particularly in   \(\bQ_v(\K_v)\). By the way the parabolic $\bQ_v$ corresponds to the set of roots $\Theta_{\Gamma,v}$ (see \cite[\S8.6]{benoist2016random}) we in particular see that \(\Theta_{\Gamma,v}\) does not include any roots coming from \(\bH_v\). In other words  \(\alpha^{\omega}(\kappa(\Gamma))\) is bounded for every simple root \(\alpha\) of \(\bH_v\).
Consequently, \(\kappa_{\bH_v}(\Gamma_v)\) is bounded, where \(\Gamma_v\) is the projection of \(\Gamma^\circ\) to \(\bH_v(\K_v)\) and \(\kappa_{\bH_v}\) is the Cartan projection of \(\bH_v(\K_v)\).
It follows that \(\Gamma_v\) is bounded. Indeed, this follows in the Archimedean case from the Cartan decomposition
\[
\bG_v(\K_v)=K_v\exp(\mathfrak a_v^{+})K_v,
\]
where \(K_v\) is a maximal compact subgroup. The non-Archimedean case is covered in  \cite{parreau2003elliptic}.
\end{proof}

We now recall Ozawa's $\mathrm P_{\mathrm{PHP}}$ condition \cite[\S8]{ozawa2025proximality}.
\begin{defn}\label{def:php}
An action of a group $\Gamma$ on a set $X$ is said to have property $\mathrm P_{\mathrm{PHP}}$ if for any finite set $F\subset\Gamma$ and every $\epsilon>0$ there exist $n\in\N$,
elements $\gamma_{1},...,\gamma_{n}\in\Gamma$, and subsets $C_{i}\subset D_{i}\subset X$
such that:
\begin{enumerate}
\item The members of
\[
\left\{ aC_{i}\mid a\in F,1\leq i\leq n\right\} \cup
\left\{ a\gamma_{i}^{-1}\left(X\backslash D_{i}\right)\mid a\in F,1\leq i\leq n\right\}
\]
are pairwise disjoint.
\item The intersection of any subcollection of
\[
\left\{ D_{i}\mid1\leq i\leq n\right\} \cup
\left\{ \gamma_{i}^{-1}\left(X\backslash C_{i}\right)\mid1\leq i\leq n\right\}
\]
of size $\geq\epsilon n^{1/2}$ is empty.
\end{enumerate}
We say that $\Gamma$ has property $\mathrm P_{\mathrm{PHP}}$ if the left regular action
$\Gamma\actson \Gamma$ (by left translation) has property $\mathrm P_{\mathrm{PHP}}$.
\end{defn}

Note that if $\Gamma$ admits some action $\Gamma\actson X$ with property $\mathrm P_{\mathrm{PHP}}$,
then $\Gamma$ has property $\mathrm P_{\mathrm{PHP}}$. Indeed, fix $x \in X$, and for $A\subset X$
let $\hat{A}:=\{\gamma\in \Gamma:\gamma.x\in A\}$. It is straightforward to verify that if
$C_i,D_i\subset X$ satisfy both conditions, then so do $\hat{C}_i$ and $\hat{D}_i$
(with the same elements $\gamma_1,...,\gamma_n \in \Gamma$).

\smallskip
In order to prove $\mathrm{P}_{\mathrm{PHP}}$ we will need the following:
\begin{prop}\label{prop:good-sequence}
Let $F\subset G$  be a finite set such that  $\ker(G\actson \mathcal B)\cap F^{-1}F=\{e\}$. Fix a finite set $E\subset\mathcal{B}$, and $n\in\N$.
Then there exists a non-empty $S$-Zariski open subset $\Omega\subset G_c^{n}$ such that every
$(g_{1},\dots,g_{n})\in\Omega$ satisfies:
\begin{enumerate}
\item The points $ag_{i}x$, running over all $a\in F\cup\{e\}$, $1\le i\le n$ and $x\in E$, are pairwise distinct.
\item For each $x\in E$, the set $\{g_{1}x,\dots,g_{n}x\}\subset\mathcal{B}$ is in general position
(Definition~\ref{def:general-position s-adic}).
\end{enumerate}
\end{prop}

\begin{proof}
Since the conditions are finite in number, it suffices to show that for each individual failure
condition, the corresponding failure locus is contained in a proper Zariski closed subset of $G^{n}_c$.
Then $\Omega$ may be taken as the complement of the (finite) union of these proper closed subsets.
Note that $G^{n}_c$ is irreducible in the $S$-Zariski topology; hence such a finite union
cannot cover $G^{n}_c$, and $\Omega$ is non-empty.

\smallskip
\noindent\emph{(1) Distinctness.}
Fix $(a,i,x)\neq(b,j,y)$ with $a,b\in F\cup\{e\}$, $1\le i,j\le n$, $x,y\in E$ and consider
\[
Z_{a,i,x}^{b,j,y}:=\{(g_{1},\dots,g_{n})\in G_c^{n}\,:\, ag_{i}x=bg_{j}y\}.
\]
This is $S$-Zariski closed. We claim it is proper.

If $i\neq j$, fix $(g_{\ell})_{\ell\neq i}$ arbitrarily. Then the condition
$ag_{i}x=bg_{j}y$ forces $g_{i}x$ to be a prescribed point of $\mathcal{B}$, hence cuts out a
proper subset of $G_c$ in the $i$-th coordinate since $\mathcal{B}$ is not a point.
If $i=j$ and $a=b$, then $ag_{i}x=ag_{i}y$ forces $x=y$, so $Z_{a,i,x}^{a,i,y}=\emptyset$ when
$x\neq y$.
If $i=j$ and $a\neq b$, then $c=b^{-1}a$ is, by assumption, not in the kernel of $G\actson \mathcal B$. Now $ag_{i}x=bg_{i}y$ becomes
$cg_{i}x=g_{i}y$, and this is avoided for some $g_{i}\in G_c$ by Lemma~\ref{lem:Machnisilizer}.
Thus $Z_{a,i,x}^{b,j,y}$ is proper in all cases.

\smallskip
\noindent\emph{(2) General position.}
Fix $x\in E$. It suffices to impose general position
separately in each coordinate $v\in S$.
Fix $v\in S$ and a pair $(I,m)$ with $I\subset\{1,\dots,n\}$ finite and $m\notin I$.
Consider the failure locus in the $v$-factor:
\[
Z_{v}(I,m;x)
:=\Bigl\{(h_{1},\dots,h_{n})\in \bG_{v}^\circ(\K_{v})^{n}\,:
\bigcap_{i\in I}\bY_{h_{i}x_{v}}\neq\emptyset\ \text{and}\
\bigcap_{i\in I}\bY_{h_{i}x_{v}}\subset \bY_{h_{m}x_{v}}
\Bigr\}.
\]
It is contained in the set
\[
\tilde{Z}_{v}(I,m;x)
:=\Bigl\{(h_{1},\dots,h_{n})\in \bG_{v}^\circ(\K_{v})^{n}\,:
\bigcap_{i\in I}\bY_{h_{i}x_{v}}\subset \bY_{h_{m}x_{v}}
\Bigr\},
\]
which is Zariski-closed (see e.g.\ \cite[Prop.\ 1.7]{borel2012linear}).
Moreover, $\tilde{Z}_{v}(I,m;x)$ (and particularly $Z_{v}(I,m;x)$)  is proper: taking $h_{i}=e$ for $i\in I$ gives
$\bigcap_{i\in I}\bY_{h_{i}x_{v}}=\bY_{x_{v}}\neq\emptyset$, and since
$\bY_{x_{v}}\subsetneq\bB_{v}$ and $\bG_v^\circ(\K_v)$ acts transitively on $\mathcal B_v$,  we can choose $h_{m}$ so that
$\bY_{x_{v}}\nsubseteq h_{m}\bY_{x_{v}}= \bY_{h_{m}x_{v}}$.

Finally, let $\pi_{v}:G_c\to \bG^\circ_{v}(\K_{v})$ be the projection and set
$\pi_{v}^{n}:G_c^{n}\to\bG^\circ_{v}(\K_{v})^{n}$.
Then $(\pi_{v}^{n})^{-1}(Z_{v}(I,m;x))$ is a proper Zariski closed subset of $G_c^{n}$.
This handles the individual $(v,I,m,x)$ failure condition.
\end{proof}

\begin{thm}\label{thm:php}
Let $S$ be a finite set such that to each $v\in S$ is associated a local field $\K_{v}$ and a semisimple   $\K_{v}$-algebraic group $\bG_{v}$. Let $G=\prod_{v\in S}\bG_{v}(\K_{v})$ and let
$\Gamma<G$ be $S$-Zariski dense. Let $\mathcal B$ be the corresponding flag space associated to $G$ and $\Gamma$, as defined in (\ref{eq:boundary}).  If the action $\Gamma\actson \mathcal B$  is faithful, then it satisfies  condition $\mathrm{P}_{\mathrm{PHP}}$.
\end{thm}

\begin{proof}
By Lemma~\ref{lem:proximal-exists}, we may fix an element $\gamma_0\in \Gamma$ such that both $\gamma_0$ and $\gamma_0^{-1}$ are $\Theta_{\Gamma}$-proximal
with attracting/repelling fixed points $(x^{+},x^{-})\in\mathcal{B}\times\mathcal{B}$.

Let $F\subset\Gamma$ be finite and let $\epsilon>0$.
Let $K=K(G)$ be the constant from Proposition~\ref{prop:large-intersections-general}.
Choose $n\in\N$ so that $\epsilon n^{1/2}>2K$.

Since $G_c$ has finite index in $G$ and $\Gamma<G$ is $S$-Zariski dense, we have that  $\Gamma_c:=\Gamma\cap G_c$ is $S$-Zariski dense in $G_c$.
The open set $\Omega\subset G_c^n$ given by Proposition~\ref{prop:good-sequence} is non-empty and $S$-Zariski open, and therefore meets the $S$-Zariski dense subset $\Gamma_c^n$. Thus we may find elements $s_{1},...,s_{n}\in\Gamma_c$
satisfying both conditions of Proposition~\ref{prop:good-sequence} with respect to the finite set
$E=\{x^{+},x^{-}\}\subset\mathcal{B}$.
Set $x_{i}^{\bullet}:=s_{i}x^{\bullet}$ for $\bullet\in\{+,-\}$.
Then the points $ax_{i}^{\bullet}$, for $a\in F\cup\{e\}$, $1\leq i\leq n$ and $\bullet\in\{+,-\}$,
are all distinct, and the sets $\{x_{1}^{+},...,x_{n}^{+}\}$ and $\{x_{1}^{-},...,x_{n}^{-}\}$ are
in general position.
In particular,
\begin{equation}\label{eq:empty-locus}
\bigcap_{i\in I}Y_{x_{i}^{+}}=\bigcap_{i\in I}Y_{x_{i}^{-}}=\emptyset
\qquad\text{for all }I\subset\{1,...,n\}\text{ with }|I|>K.
\end{equation}
where, once again, $K=K(G)$ is the constant provided by Proposition \ref{prop:large-intersections-general}.

Since each $\bB_{v}$ is projective, $\mathcal{B}_{v}=\bB_{v}(\K_{v})$ is compact, hence so is
$\mathcal{B}=\prod_{v\in S}\mathcal{B}_{v}$.
For each $i$ choose an open neighborhood $U_{i}^{\bullet}\subset\mathcal{B}$ of $x_{i}^{\bullet}$ and an open neighborhood $V_{i}^{\bullet}\subset\mathcal{B}$ of the closed set $Y_{x_{i}^{\bullet}}$ containing $U_i^\bullet$. These sets should be chosen sufficiently small so that:
\begin{enumerate}
\item the disjointness of the points $ax_{i}^{\bullet}$ implies that the sets $aU_{i}^{\bullet}$
(for $a\in F\cup\{e\}$, $1\le i\le n$, $\bullet\in\{+,-\}$) are pairwise disjoint;
\item for every $I\subset\{1,...,n\}$ with $|I|>K$ one has
$\bigcap_{i\in I}V_{i}^{+}=\bigcap_{i\in I}V_{i}^{-}=\emptyset$. In particular, the intersection of any subcollection of $\{V_i^+,V_i^-\}_{i=1}^n$ of size $>2K$ is empty.
\end{enumerate}
For each $i$, the element $s_{i}\gamma_{0}s_{i}^{-1}$ is $\Theta_{\Gamma}$-proximal with
attracting/repelling fixed points $(x_{i}^{+},x_{i}^{-})$.
By Fact~\ref{fact:loxodromic} applied to $s_{i}\gamma_{0}s_{i}^{-1}$ and its inverse, we may choose
$m_{i}\in\N$ sufficiently large so that
\[
\gamma_{i}:=s_{i}\gamma_{0}^{m_{i}}s_{i}^{-1}
\]
satisfies
\[
\gamma_{i}\bigl(\mathcal{B}\setminus V_{i}^{-}\bigr)\subset U_{i}^{+}
\qquad\text{and}\qquad
\gamma_{i}^{-1}\bigl(\mathcal{B}\setminus V_{i}^{+}\bigr)\subset U_{i}^{-}.
\]
Equivalently,
\[
\gamma_{i}^{-1}\bigl(\mathcal{B}\setminus U_{i}^{+}\bigr)\subset V_{i}^{-}
\qquad\text{and}\qquad
\gamma_{i}\bigl(\mathcal{B}\setminus U_{i}^{-}\bigr)\subset V_{i}^{+}.
\]
Condition $\mathrm{P}_{\mathrm{PHP}}$ thus holds with $C_i=U_i^+$ and $D_i=V_i^+$.
\end{proof}

The generality of Theorem~\ref{thm:php} is essential for deducing $\mathrm P_{\mathrm{PHP}}$ for arbitrary linear groups with trivial amenable radical. As a special case, we get one of the theorems from the introduction.

\begin{proof}[Proof of Theorem \ref{thm:s-adic}]
     $\Gamma$ is $S$-Zariski dense by Lemma \ref{lem:S-Zariski-density}, and its projection into each simple factor $\bG(\K_v)$ is unbounded by assumption. It follows from Lemma \ref{lem:proximal-exists} that the $G$-action on $\mathcal B$ is faithful, and particularly the $\Gamma$-action. Theorem \ref{thm:php} thus applies, and so $\Gamma$ has property $\mathrm{P}_{\mathrm{PHP}}$. The corresponding result about twisted reduced group $C^*$-algebra stated in Theorem \ref{thm: twisted} follows from \cite[Theorem A]{flores2026pureness}.
\end{proof}

\section{Linear groups\label{sec:Linear-groups}}

The purpose of this section is to prove Theorem~\ref{thm:main} by reducing it to
Theorem~\ref{thm:php}.

We begin by recalling that torsion in linear groups is bounded, provided that the
field is finitely generated. The finite-generation assumption is necessary: for
example, the multiplicative group of the field \(\bar{\Q}\) contains roots of
unity of arbitrarily large order, and therefore has unbounded torsion.

\begin{lem}\label{lem:torsion}
Let $k$ be a finitely generated field, and let $d\in\N$.
There exists $m=m(k,d)$ such that every element $g\in\GL_{d}(k)$ of finite order
satisfies $g^{m}=1$.
\end{lem}

\begin{proof}
We first treat unipotent torsion. Let $u\in\GL_d(k)$ be unipotent and torsion.
If $\mathrm{char}(k)=0$ then $u=1$. If $\mathrm{char}(k)=p>0$, write $u=1+N$ with $N$
nilpotent and $N^{d}=0$. Choose $e\in\N$ such that $p^{e}\ge d$. Since
$\binom{p^{e}}{j}\equiv 0\pmod p$ for $0<j<p^{e}$, we have
\[
u^{p^{e}}=(1+N)^{p^{e}}=1+N^{p^{e}}=1,
\]
because $N^{p^{e}}=0$. Thus there is an integer $q=q(d,k)$ such that every unipotent
torsion element satisfies $u^{q}=1$ (one may take $q=1$ in characteristic $0$, and
$q=p^{e}$ in characteristic $p>0$ with $p^{e}\ge d$).

Next, by \cite[Lemma~2.3]{tits1972free} there are only finitely many roots of unity
$x\in\bar{k}$ that are algebraic over $k$ of degree $\le d$. Let $r=r(k,d)$ be the least
common multiple of their orders, so that every such $x$ satisfies $x^{r}=1$.

Let $g\in\GL_d(k)$ be any torsion element. Consider $g$ in $\GL_d(\bar{k})$ and write
its multiplicative Jordan decomposition $g=su$, where $s$ is semisimple, $u$ is unipotent,
and $su=us$. Then $u$ is torsion, hence $u^{q}=1$. The eigenvalues of $s$ coincide with
the eigenvalues of $g$. Hence each
eigenvalue $\lambda$ of $s$ is a root of unity and is algebraic over $k$ of degree $\le d$,
so $\lambda^{r}=1$. Therefore $s^{r}=1$.

With $m:=rq$ we get
\[
g^{m}=(su)^{rq}=s^{rq}u^{rq}=1,
\]
as required.
\end{proof}

\begin{lem}\label{lem:zariski-semisimple-implies-trivial-ar}
Let $k$ be an arbitrary field. Let  $\Gamma\le\GL_{d}(k)$ be a finitely generated subgroup, and let $\bG:=\overline{\Gamma}^{\,z}$
be its Zariski closure.
Assume that $\bG^\circ$ is semisimple and that 
\[
\Gamma\cap C_{\bG}(\bG^\circ)=\{e\}
\]
where $C_{\bG}(\bG^\circ)$ is the centralizer of $\bG^\circ$ in $\bG$.
Then $\Gamma$ has trivial amenable radical.
\end{lem}

\begin{proof}
Let $A:=\Rad(\Gamma)$ and let $\bH:=\overline{A}^{\,z}\le\bG$ be its Zariski closure.
Since $A$ is normal in $\Gamma$ and $\Gamma$ is Zariski dense in $\bG$, the subgroup $\bH$ is normal
in $\bG$. In particular, $\bH^\circ$ is a connected normal algebraic subgroup of the semisimple group
$\bG^\circ$, and hence $\bH^\circ$ is semisimple.

Since $\Gamma$ is finitely generated, all elements of $\Gamma$ lie in $\GL_d(k_0)$ for some finitely
generated subfield $k_0\subset k$.
As $A$ is amenable, it contains no nonabelian free subgroup. By \cite[Theorems~1 and~2]{tits1972free},
there exists a normal solvable subgroup $R\triangleleft A$ such that $A/R$ is locally finite.
Let $\overline R^{\,z}$ be the Zariski closure of $R$ in $\bH$.
Since $R\triangleleft A$ and $A$ is Zariski dense in $\bH$, the subgroup $\overline R^{\,z}$ is normal
in $\bH$. Thus $(\overline R^{\,z})^\circ$ is a connected solvable normal algebraic subgroup of the
semisimple group $\bH^\circ$, and hence is trivial. Therefore $\overline R^{\,z}$ is finite, and in
particular $R$ is finite.

It follows that $A$ is locally finite. Indeed, $A/R$ is locally finite and $R$ is finite.
By Lemma~\ref{lem:torsion}, there exists $m>0$ such that $a^{m}=e$ for all $a\in A$.
This polynomial equation must therefore hold for all elements in the Zariski closure $\bH$.
If $\bH^\circ$ were nontrivial, then, after passing to the algebraic closure $\bar{k}$, the group $\bH^\circ$ would contain a positive-dimensional torus $\mathrm{GL}_1\leq \bH^\circ$. It would then follow that $x^m=1$ for all $x\in\bar{k}^*$, which is of course impossible. Thus $\bH^\circ$ is trivial, so $\bH$ is finite.

Finally, since $\bH$ is a finite normal algebraic subgroup of $\bG$, the connected group $\bG^\circ$
acts trivially on $\bH$ by conjugation. Hence $\bH\le C_{\bG}(\bG^\circ)$. Therefore
\[
A\le \Gamma\cap C_{\bG}(\bG^\circ)=\{e\}.
\]
Thus $A=\{e\}$.
\end{proof}

\begin{lem}\label{lem:amenable-radical}
Let $\Gamma$ be a linear group with trivial amenable radical.
Then $\Gamma$ is a directed union of finitely generated subgroups with trivial amenable radicals.
\end{lem}

\begin{proof}
Fix an embedding $\Gamma\leq\mathrm{GL}_{d}(K)$ for some algebraically closed field $K$, and let
$\bG$ denote the Zariski closure of $\Gamma$.
Let $\bR$ be the solvable radical of $\bG^\circ$.
Since $\bR$ is normal in $\bG$, the group $\Gamma\cap\bR$ is a normal amenable subgroup of $\Gamma$.
Thus $\Gamma\cap\bR=\{e\}$. After quotienting by $\bR$, we may therefore assume that $\bG$ is
semisimple.

Write $\Gamma=\bigcup_E \Gamma_E$, where $\Gamma_E$ is the finitely
generated subgroup generated by a finite set $E\subset\Gamma$. Let
$\bH_E\leq\bG$ denote the Zariski closure of $\Gamma_E$ in $\bG$.
Let $\bH_E^\circ$ denote the connected component of the identity,
and consider the upward directed diagram of connected algebraic subgroups
$\bH_E^\circ$. Since an increasing chain of connected algebraic subgroups must have strictly increasing dimensions until it stabilizes, there exist an algebraic subgroup $\bH\leq\bG$ and a finite set $E_0\subset\Gamma$ such that $\bH_E^\circ=\bH$ for all finite $E\supset E_0$.

For every finite $E\supset E_0$, the subgroup $\bH=\bH_E^\circ$ is normal
in $\bH_E$. Thus the equation $g\bH g^{-1}=\bH$ holds for every
$g\in\bigcup_{E\supset E_0}\bH_E$, and in particular for every $g\in\Gamma$.
By Zariski density, it therefore holds for every $g\in\bG$. We thus see that
$\bH\leq\bG$ is normal, and denote the quotient map by
$\pi:\bG\to\bG/\bH$. With this notation, we note that
$\pi(\Gamma_E)\leq \pi(\bH_E)=\bH_E/\bH$ is finite for every finite
$E\supset E_0$. Therefore
$\pi(\Gamma)=\bigcup_{E\supset E_0}\pi(\Gamma_E)$ is locally finite, and in
particular amenable.

Consider
\[
C:=\Gamma\cap C_{\bG}(\bH).
\]
Since $\bH$ is normal in $\bG$, the algebraic subgroup $C_{\bG}(\bH)$ is normal in $\bG$, and therefore
$C$ is normal in $\Gamma$. Moreover, the restriction of $\pi$ to $C$ is injective. Indeed,
\[
\ker(\pi|_C)=C\cap \bH=\Gamma\cap Z(\bH).
\]
Since $\bH$ is connected semisimple, $Z(\bH)$ is finite; since $\bH$ is normal in $\bG$, 
$Z(\bH)$ is normal in $\bG$. Hence $\Gamma\cap Z(\bH)$ is a finite normal subgroup of $\Gamma$, and is
therefore trivial. Thus $C$ embeds into the locally finite group $\pi(\Gamma)$, and hence $C$ is
amenable. We get that $C$ is a normal amenable subgroup of $\Gamma$, and is therefore trivial.
Hence
\[
\Gamma_E\cap C_{\bH_E}(\bH_E^\circ)
\subseteq
\Gamma\cap C_{\bG}(\bH)
=C=
\{e\}.
\]

Lemma~\ref{lem:zariski-semisimple-implies-trivial-ar}, applied to the finitely generated group
$\Gamma_E$ and its Zariski closure $\bH_E$, shows that $\Gamma_E$ has trivial amenable radical for all finite $E\supset E_0$.
We conclude that the subgroups $\{\Gamma_E\}_{E\supset E_0}$ form an upward directed family of finitely generated subgroups with trivial amenable radicals, and their  union is
$\Gamma$. This proves the lemma.
\end{proof}

We need one more fact.
\begin{lem}\label{lem:directed-union-Cr-twisted}
Let $\{\Gamma_i\}_{i\in I}$ be an upward directed family of groups and set
$\Gamma=\bigcup_{i\in I}\Gamma_i$.
Let $\omega\in Z^2(\Gamma,\T)$ be a normalized $2$-cocycle, and write
$\omega_i=\omega|_{\Gamma_i\times \Gamma_i}$.
Then the inclusions $\Gamma_i\le \Gamma$ induce canonical embeddings
$
C_r^*(\Gamma_i,\omega_i)\ \hookrightarrow\ C_r^*(\Gamma,\omega)
$
such that
\[
C_r^*(\Gamma,\omega)\ =\ \overline{\bigcup_{i\in I} C_r^*(\Gamma_i,\omega_i)}.
\]
In particular, for $\omega=1$ we have
\[
C_r^*(\Gamma)\ =\ \overline{\bigcup_{i\in I} C_r^*(\Gamma_i)}.
\]
\end{lem}

\begin{proof}
The inclusion $C_r^*(\Gamma_i,\omega_i)\hookrightarrow C_r^*(\Gamma,\omega)$ is standard; see e.g.\ \cite[\S4]{bedos2023c}. Since $\Gamma=\bigcup_i\Gamma_i$, we have
$\C[\Gamma]=\bigcup_i\C[\Gamma_i]$.
Taking closures in $B(\ell^2(\Gamma))$ yields
\[
C_r^*(\Gamma,\omega)=\overline{\bigcup_{i\in I} C_r^*(\Gamma_i,\omega_i)}.\qedhere
\]
\end{proof}

\begin{proof}[Proof of Theorem~\ref{thm: twisted} (and particularly Theorem~\ref{thm:main})]
Let $\Gamma$ be a linear group with trivial amenable radical.
By Lemma~\ref{lem:amenable-radical}, $\Gamma$ is a directed union of finitely generated subgroups
$\{\Gamma_i\}$ with trivial amenable radical.
By Lemma~\ref{lem:directed-union-Cr-twisted},
\[
C_{r}^{*}(\Gamma,\omega)=\overline{\bigcup_{i}C_{r}^{*}(\Gamma_{i},\omega\!\mid_{\Gamma_i})}
\]
for any $2$-cocycle $\omega$ on $\Gamma$.
Invoking \cite[Theorem~4.1]{robert2025selfless}, we reduce to proving that each
$C_{r}^{*}(\Gamma_{i},\omega\!\mid_{\Gamma_i})$ is selfless. Thus, we lose no generality by assuming that
$\Gamma$ is finitely generated.

We will show that $\Gamma$ satisfies $\mathrm P_{\mathrm{PHP}}$. This will in turn imply that every twisted reduced group
$C^*$-algebra of $\Gamma$ (and in particular the untwisted algebra $C_r^*(\Gamma)$) is selfless, by \cite[Theorem~14]{ozawa2025proximality} and,
more generally, \cite[Theorem~A]{flores2026pureness}.

Fix an embedding $\Gamma\le\GL_{d}(k)$ for some field  $k$. Since $\Gamma$ is finitely generated, we may assume $k$ is a finitely generated field. Let
$\bG:=\overline{\Gamma}^{\,z}$ be the Zariski closure of $\Gamma$.
As in the proof of Lemma~\ref{lem:amenable-radical}, after passing to a suitable quotient of $\bG$
we may assume that $\bG$ is semisimple and has no nontrivial finite normal algebraic subgroups.

Let $\bG^{\circ}$ be the identity component. Then $\bG^\circ$ is connected, semisimple, and adjoint, and as such it admits a decomposition \cite[Proposition 14.10]{borel2012linear}
\[
\bG^\circ=\prod_{j\in J}\bH_{j}
\]
of $k$-simple factors, where $J$ is a finite indexing set.
After replacing $k$ by a finite field extension (and replacing the $\bH_{j}$ accordingly), we may
assume that each $\bH_{j}$ is absolutely simple.

Let $\Gamma_c=\Gamma\cap \bG^\circ$.
For each $j\in J$, the projection of $\Gamma_c$ to $\bH_{j}(k)$ is Zariski dense, hence infinite.
There must therefore exist a regular function $f_{j}\in k[\bH_{j}]$ such that
$f_{j}(\Gamma_c)$ is infinite (e.g. a matrix coefficient of $\mathrm{GL}_d$).
By \cite[Lemma~2.1]{breuillard2007topological}, there exists a local field $\K_{j}$ and a field
embedding $\iota_{j}:k\hookrightarrow \K_{j}$ such that $\iota_{j}(f_{j}(\Gamma_c))$ is unbounded in
$\K_{j}$. In particular, the image of $\Gamma_c$ in $\bH_{j}(\K_{j})$ is unbounded.

We now regard the finite set $J$ as the set indexing the factors of an $S$-algebraic group. Thus, for convenience, from this point on we write $S:=J$, and relabel $\bH_{j}$, $\K_{j}$ and $\iota_j$ as $\bH_v$, $\K_v$ and $\iota_v$, respectively, for $v\in S$.

Consider the $S$-algebraic group
\[
G:=\prod_{v\in S}\bG(\K_{v}),\qquad G_c:=\prod_{v\in S}\bG^\circ(\K_{v}),
\]
and use the embeddings $\iota_{v}$ to view $\bG(k)$, and in particular $\Gamma$, as a subgroup of $G$.
Now $\Gamma_c$ is Zariski dense in $\bG^\circ(k)$, hence also in $\bG^\circ(\K_v)$ for each $v\in S$ \cite[Cor.\ 18.3]{borel2012linear}.
Thus, by Lemma~\ref{lem:S-Zariski-density}, $\Gamma_c$ is $S$-Zariski dense in $G_c$.

We consider the $G$-space $\mathcal B$ as defined in \eqref{eq:boundary}. We will show that
the $\Gamma$-action on $\mathcal B$ is faithful (note however that the $G$-action is not necessarily faithful), and then we will be done, since Theorem~\ref{thm:php} applies.  
The rest of the proof is therefore devoted to showing that $\bG(k)$ acts faithfully on $\mathcal B$.

Let $N$ be the kernel of the action $\bG(k)\actson \mathcal B$.
For each $v\in S$, let $\bN_v\leq \bG_{\K_v}$ denote the kernel of the
$\K_v$-algebraic action $\bG_{\K_v}\actson \mathcal B_v$.
Thus $\bN_v$ is a normal $\K_v$-algebraic subgroup of $\bG_{\K_v}$.
Since an element acting trivially on $\mathcal B$ acts trivially on each
factor $\mathcal B_v$, we have
\[
N\subseteq \bN_v(\K_v)\cap \bG(k)\qquad\text{for every }v\in S,
\]
where $\bG(k)$ is viewed inside $\bG(\K_v)$ through the embedding
$\iota_v:k\hookrightarrow \K_v$.

Assume for contradiction that $N$ is infinite, and let
$\bM:=\overline{N}^{\,z}$ be its $k$-Zariski closure in $\bG$.
Since $N$ is normalized by $\Gamma$, Zariski density implies that $\bM$ is a
normal algebraic $k$-subgroup of $\bG$. Moreover, $\bM$ is positive-dimensional,
and hence $\bM^\circ$ is a nontrivial connected normal algebraic subgroup of
$\bG^\circ$. Since $\bG^\circ=\prod_{v\in S}\bH_v$ is a direct product of
simple adjoint factors, there exists $v_0\in S$ such that
$\bH_{v_0}\leq \bM^\circ$.

Now pass to the local field $\K_{v_0}$. Since $N$ is
$k$-Zariski dense in  $\bM$, it is $\K_{v_0}$-Zariski dense in the $\K_{v_0}$-group
$\bM_{\K_{v_0}}$ obtained from $\bM$ via extension of scalars. As $N\subseteq \bN_{v_0}(\K_{v_0})$, we see that
$\bM_{\K_{v_0}}\leq \bN_{v_0}$, and in particular
$(\bH_{v_0})_{\K_{v_0}}\leq \bN_{v_0}$. Thus $\bH_{v_0}(\K_{v_0})$ acts
trivially on $\mathcal B_{v_0}$.

It follows from Lemma~\ref{lem:proximal-exists} that the projection of
$\Gamma_c$ to $\bH_{v_0}(\K_{v_0})$ is bounded, contradicting the choice of
$\K_{v_0}$. Therefore $N$ must be finite. Since $\bG$ has no nontrivial finite
normal algebraic subgroups, we conclude that $N=\{e\}$.
\end{proof}

\smallskip
\noindent\textit{Remark.} 
If $\bG$ were connected, the argument would be simpler: one could set
$G=\prod_j \bH_j(\K_j)$ and apply Theorem~\ref{thm:php} directly to the faithful action on the corresponding boundary. It is tempting to reduce to the connected case by replacing $\Gamma$ with $\Gamma\cap \bG^\circ$, but this would require a permanence result for selflessness under finite extensions. We expect that, if $\Lambda<\Gamma$ is a finite-index subgroup, then selflessness of $C_r^*(\Gamma)$ implies selflessness of $C_r^*(\Lambda)$, and conversely that selflessness of $C_r^*(\Lambda)$ implies selflessness of $C_r^*(\Gamma)$ provided $\Gamma$ is ICC. However, we were unable to prove this permanence statement. The analogous permanence statement for $C^*$-simplicity is known; see \cite[Proposition 19]{delaHarpe2007simplicity} and, more generally, \cite[\S4.7]{popa2000relative}.

\bibliographystyle{alpha}
\bibliography{ProjectSelfless}

\vspace{0.5cm}

\noindent{\textsc{Department of Mathematics, University of California San Diego, 9500 Gilman Drive, La Jolla, CA
92093, USA}}
\vspace{0.5cm}

\noindent{\textit{Email address:} \texttt{ivigdorovich@ucsd.edu}}

\noindent{\textit{Webpage:} \texttt{https://sites.google.com/view/itamarv}} \\
\end{document}